\documentclass[12pt]{amsart}

\usepackage{cite}

\usepackage{amsmath}
\usepackage{amstext}
\usepackage{amssymb}

\usepackage{esint}

\usepackage{amsthm}

\theoremstyle{plain}
\newtheorem{theorem}{Theorem}[section]
\newtheorem{lemma}[theorem]{Lemma}
\newtheorem{corollary}[theorem]{Corollary}

\newtheorem{refthm}{Theorem}

\theoremstyle{definition}

\numberwithin{equation}{section}

\DeclareMathOperator{\Rp}{Re}

\newcommand{\bp}{\mathcal{P}}

\usepackage{cite}

\begin{document}

\title[Extremal Problems and Square Functions]
{Bergman and Szeg\"{o} projections, Extremal Problems, and Square Functions}
\author{Timothy Ferguson}
\address{Department of Mathematics\\University of Alabama\\Tuscaloosa, AL}
\email{tjferguson1@ua.edu}

\date{\today}

\begin{abstract}
  We study estimates for Hardy space norms of analytic projections.
  We first find a sufficient condition for the Bergman projection of a
  function in the unit disc to belong to the Hardy space $H^p$ for
  $1 < p < \infty$.  We apply the result to prove a converse to
  an extension of Ryabykh's theorem about Hardy space regularity
  for Bergman space extremal functions.
  We also prove that the $H^q$ norm of the Szeg\"{o} projection of
  $F^{p/2} \overline{F}^{(p/2)-1}$ cannot be too small if $F$ is
  analytic, for certain values of $p$ and $q$.
  We apply this to show that the best analytic approximation
  in $L^p$ of
  a function in both $L^p$ and $L^q$ will also lie
  in $L^q$, for certain values of $p$ and $q$.  
\end{abstract}




\maketitle

\section{Introduction}
This paper deals with estimates for
Hardy space norms of analytic projections.  One of the main tools
is a result (due to Calderon) about square functions of
functions of the form $|F|^\delta$, where $F$ is analytic in the
unit disc.
We find a sufficient condition for the Bergman projection of a function
to lie in a Hardy space.  
We also apply our results to extremal problems in spaces of
analytic functions.

We begin with some background. 
Let $\bp(f)$ denote the Bergman projection of a
function $f \in L^1(\mathbb{D})$.  This can be defined by the formula
\[
  \bp(f)(z) = \int_{\mathbb{D}} f(w) \frac{1}{(1-z\overline{w})^2}
  \frac{dA(w)}{\pi}.
\]
The Bergman projection is the projection of
$L^2(\mathbb{D}, dA/\pi)$ onto the
Bergman space $A^2(\mathbb{D})$ consisting of all analytic functions in
$L^2(\mathbb{D})$.  The Bergman spaces $A^p$ for $1 < p < \infty$
are defined similarly.  

There are many results about when a Bergman projection lies in
certain weighted Lebesgue or Sobolev spaces, but we want to find
a condition for a Bergman projection to lie in a Hardy space.
Recall that $f$ is in the Hardy space $H^p$ if it is analytic in the
unit disc and
\[
  \|f\|_{H^p} = \sup_{0 \leq r < 1} M_p(r,f) < \infty
\]
where
\[
  M_p(r,f) = \left\{ \frac{1}{2\pi} \int_0^{2\pi}
    |f(re^{i\theta})|^p \, d\theta \right\}^{1/p}.
\]

For a function $f \in W^{1,2}_{\textrm{loc}}(\mathbb{D})$ we may define
\[
  S_z(f)(e^{i\theta})^2 =
  \iint_{\Gamma(\theta)} |\partial_z f(re^{i\phi})|^2 \,
                       \, dr \, d\theta,
\]
where $\Gamma_{\theta}$ is a cone centered at $e^{i \theta}$, to be
specified later.       
For an analytic function, this is equivalent to the Zygmund $S$
function. Our first result says that, under suitable conditions,
if $S_z(f) \in L^p(\partial \mathbb{D})$ then $\bp(f) \in H^p$ for
$1 < p < \infty$.  

An application is as follows.  Let $1 < p < \infty$.  Then
the dual space 
$(A^p)^*$ is isomorphic to $A^{p'}$.  Given a nonzero $k \in A^{p'}$,
we say the extremal function $F$ associated with the integral kernel
$k$ is the function of unit $A^p$ norm such that 
\[
  \Rp \frac{1}{\pi} \int_{\mathbb{D}} F \overline{k} \, dA =
  \sup_{f \in A^p,\ \|f\|=1}
  \Rp \frac{1}{\pi} \int_{\mathbb{D}} f \overline{k} \, dA.
\]
It is know that $F$ exists and is unique.  Also,
\[
  k = \|k\|_{(A^p)^*} \bp(F^{p/2} \overline{F}^{(p/2)-1}).
\]

One important result about these problems is that 
if $k \in H^q$ for some $q$ such that $p' \leq q < \infty$ then
$F \in H^{(p-1)q}$.  This was first proved by Ryabykh\cite{Ryabykh} in
the case $q=p'$ for $1 < p < \infty$.  The full result was proved by
Ferguson\cite{tjf2, tjf-ryabpnoteven}. 

It was previously known that the converse of this statement holds
for $p$ an even integer\cite{tjf2}.  In fact, if $0 < q < \infty$
and $p$ is an even integer and
$F \in H^{(p-1)q}$ then
$k \in H^q$.  We prove that this converse
result in fact holds for $1 < p < \infty$.

For our next result, we prove that for certain $p$ and $q$, there is a
constant $0 \leq C < 1$ such that $\|P_Sf\|_q \geq (1-C)\|f\|_q$, where
$f = F^{p/2} \overline{F}^{(p/2)-1}$ for analytic $F$ and $P_S$ is the
Szeg\"{o} projection.  The Szeg\"{o} projection is the orthogonal projection
of $L^2(\partial \mathbb{D})$ onto $H^2$.
We then apply this result to extremal problems for Hardy spaces.
It is known that $(H^p)^* = H^{p'}$ for $1 < p < \infty$.  As before,
we can define a Hardy space extremal function $F$ for $k \in H^{p'}$ by 
\[
  \frac{1}{2\pi} \Rp \int_{\mathbb{D}} F \overline{k} \, d\theta =
  \sup_{f \in A^p,\ \|f\|=1}
  \frac{1}{2\pi} \Rp \int_{\mathbb{D}} f \overline{k} \, d\theta.
\]
As before, the extremal function
will exist and be unique for $1 < p < \infty$ and
\[
  k = \|k\|_{(H^p)^*} P_S (F^{p/2} \overline{F}^{(p/2)-1}).
\]
Using this with our result shows that if
$k \in H^q \cap H^{p'}$ and $F$ is the $H^p$ extremal function for
$k$, then for certain values of $p$ and $q$ we have that
$F \in H^{q(p-1)}$.

We then apply the theory of dual extremal problems developed
by S.\ Ya.\ Khavinson
\cite{KhavinsonSYa1949, KhavinsonSYa1951, KhavinsonSYa1951Trans},
and Rogosinksi and Shapiro \cite{RogosinskiShapiro}
(see \cite[Chapter 8]{D_Hp} for a good introduction).
One of the results from this theory is that for $k \in L^{p'}$ and 
$1 < p < \infty$ we have
\[
  \max_{F \in H^p, \|F\|_p=1} \Rp \frac{1}{2\pi}
  \int_0^{2\pi} F(e^{i\theta}) k(e^{i\theta}) \,d \theta
  =
  \min_{g \in H_0^{p'}} \|k - g\|_{p'}
\]
where both the minimum and the maximum are attained, and
$H_0^{p'}$ consists of
those functions in $H^{p'}$ that vanish at the origin.
Moreover,
\[
  k - g = \|k\|_{(H^p)^*} F^{p/2} \overline{F}^{(p/2)-1}
\]
by the
conditions for equality in H\"{o}lder's inequality. The function
$k-g$ is called the extremal kernel.

Using this theory of dual extremal problems, we derive a 
theorem on best approximation of functions in $L^p$ spaces.  Namely, for
certain values of $p$ and $q$, if $k \in H^q \cap H^p$ and
$f$ is the analytic function in $H^p$ that is closest to
$k$, then $f \in H^q$ as well.  

We now state the result of Calderon that we will use. Let

\[
  S(G)(e^{i\theta})^2 = \iint_{\Gamma_\theta} |\nabla G(re^{i\phi})|^2 \,
                       \, dr \, d\theta
\]
where $\Gamma_{\theta}$ is the cone
$\{(r,\phi): |\theta-\phi| < (1-r)/2, 0 < r  < 1 \}$.  
\begin{refthm}\label{thm:caldhl}
  Let $F$ be analytic and let $G = |F|^\delta$ for $\delta > 0$.
  Let $0 < p < \infty$.  Then there are constants
  $C_{p, \delta}$ and $\widehat{C}_{p, \delta}$ such that 
  \[
    \|S(G)\|_p \leq C_{p, \delta} \|G\|_p
  \]
  and
  \[
    \|G\|_p \leq \widehat{C}_{p, \delta} \|S(G)\|_p \text{ if $F(0)=0$}.
  \]
\end{refthm}
We remark that the constants in this theorem may be explicitly
estimated by following Calderon's proof.

\section{Bergman Projections in Hardy Spaces}
Our first result is as follows.  
\begin{theorem}\label{thm:bpinhardy}
  If $1 < p < \infty$ and $f \in W^{1,1}(\mathbb{D}) \cap L^{2p/(p+1)}(\mathbb{D})$ and
  $M_1(r,f) = o((1-r)^{-1})$ and $S_z(f) \in L^p$, then
  $\bp f \in H^p$ and
\[
  \|\bp(f)\|_{H^p} \leq C_p (\|S_z(f)\|_p + \|f\|_{L^{2p/(p+1)}(\mathbb{D})})
\]
\end{theorem}
\begin{proof}
Let $h$ be an analytic polynomial
of degree at most $n$.  Let
$\bp_n$ be the orthogonal projection from $L^2$ onto the polynomials in
$z$ of degree at most $n$.  The Cauchy-Green theorem shows that 
\[
  \frac{i}{2\pi} \int_{0}^{2\pi} \bp_n(f) \overline{h} \, d\theta =
  \frac{1}{2\pi} \int_{0}^{2\pi} \bp_n(f) \overline{zh} \, dz =
  \frac{1}{\pi} \int_{\mathbb{D}} \bp_n(f) \overline{(zh)'}\, dA(z).
\]
But since $(zh)'$ has degree at most $n$,
the right side of the above equals
\[
  \frac{1}{\pi} \int_{\mathbb{D}} f \overline{(zh)'}\, dA(z).
\]
Now we may use the product rule and
integrate one of the resulting integrals
by parts to see that the above expression
equals
\[
  \frac{1}{\pi} \int_{\mathbb{D}} \partial_z(f) \overline{h'(z)}
  (1-|z|^2) \, dA(z) + \frac{1}{\pi} \int_{\mathbb{D}} f \overline{h} \, dA(z).
\]
Technically, to prove the above equality, we apply integration by
parts over a disc of radius $r$ and let $r \rightarrow 1$, using
the fact that $M_1(r,f) = o((1-r)^{-1})$.  

By writing the integral on the left above in
polar coordinates  we see that it equals
\[
  \begin{split}
  &\int_0^1 \frac{1}{2\pi} \int_0^{2\pi} \partial_z(f(re^{i\theta}))
    \overline{h'(re^{i\theta})}
    (1+r) r \, \int_{\theta-(1-r)/2}^{\theta+(1-r)/2} \, d\phi \, dr \, d\theta
   \\ =
  &\int_0^1 \frac{1}{2\pi} \int_0^{2\pi} 
   \int_{\phi-(1-r)/2}^{\phi+(1-r)/2} \,
   \partial_z(f(re^{i\theta})) \overline{h'(re^{i\theta})}
   (1+r) r \, d\theta \, dr \, d\phi \\
   =&
   \frac{1}{2\pi} \int_0^{2\pi} \iint_{\Gamma_\phi}
   \partial_z(f(re^{i\theta})\overline{h'(re^{i\theta})} (1+r)r\, dr \,
    d\theta \, d\phi
\end{split}  
\]
where the region \[
  \Gamma_\phi =
  \{(r,\theta):\phi - (1-r)/2 < \theta < \phi + (1-r)/2, 0<r<1\}.
  \]
  Now we may apply the Cauchy-Schwarz inequality and the fact that
  $r(1+r) \leq 2$ to deduce the
following inequality: 
\begin{equation}\label{eq:gpolarest}
  \left|  \frac{1}{2\pi} \int_0^{2\pi}  \bp_n(f) \overline{h} \, d\theta
  \right|
  \leq
  \frac{2}{2 \pi} \int_0^{2\pi} S_z(f) S_z(h) \, d\theta +
  \|f\|_{L^{2p/(p+1)}(\mathbb{D})} \|h\|_{A^{2p'}}
\end{equation}
where
\[
  S_z(f)(e^{i\theta})^2 =
  \iint_{\Gamma_\theta} |\partial_z f(re^{i\phi})|^2 \,
                       \, dr \, d\phi.
                     \]
We now see that 
\[
  \begin{split}
    \left|  \frac{1}{2\pi} \int_0^{2\pi}  \bp_n(f) \overline{h} \, d\theta
  \right| 
  &\leq
  2 \|S_z(f)\|_p \|S_z(h)\|_{p'} +   \|f\|_{L^{2p/(p+1)}(\mathbb{D})} \|h\|_{A^{2p'}} \\
  &\leq C (\|S_z(f)\|_p + \|f\|_{L^{2p/(p+1)}(\mathbb{D})}) \|h\|_{H^{p'}}  
  \end{split}
\]
by Theorem \ref{thm:caldhl} and the fact that
$\|h\|_{A^{2p'}} \leq \|h\|_{H^{p'}}$\cite{Dragan_isoperimetric}. 
We may replace $h$ by any polynomial
$\widehat{h}$ whose first $n$ terms agree with
$h$ without changing the value of the left hand side of the above inequality.
We also have that $\|\widehat{h}\|_{H^{p'}} \leq C \|h\|_{H^{p'}}$ where
$C$ depends only on $q$, 
due to the fact that Fourier partial summation operators
are uniformly bounded on $L^q$ for $1 < q < \infty$.  Taking the supremum
over all $h$ with $\|h\|_{H^{p'}} \leq 1$ thus shows that
\[
  \sup_{\|h\|_{H^{p'}} \leq 1}
  \left|  \frac{1}{2\pi} \int_0^{2\pi}  \bp_n(f) \overline{h} \, d\theta
  \right| \leq C (\|S_z(f)\|_p + \|f\|_{L^{2p/(p+1)}(\mathbb{D})}).
\]
But the boundedness of the Szeg\"{o} projection (or equivalently
the fact that
$(H^{p})^*$ is isomorphic to $H^{p'}$ for $1 < p < \infty$)
now implies that
\[
  \|\bp_n(f)\|_p \leq C (\|S_z(f)\|_p + \|f\|_{L^{2p/(p+1)}(\mathbb{D})}).
\]
Since integral means are increasing, this means that for any fixed $r$,
\[
  M_p(r,\bp_n(f)) \leq C (\|S_z(f)\|_p + \|f\|_{L^{2p/(p+1)}(\mathbb{D})}).
\]
Since $\bp_n(f)$ converges to $\bp(f)$ uniformly on compact subsets,
this implies that
\[
  M_p(r,\bp(f)) \leq C (\|S_z(f)\|_p + \|f\|_{L^{2p/(p+1)}(\mathbb{D})}). 
\]
\end{proof}
We note that a similar argument to the one above shows that
\[
   \frac{1}{2\pi} \int_0^{2\pi} P_n(\overline{z}f) \overline{h} \, d\theta 
   \leq C \|S_z(f)\|_p \|S_z(zh)\|_{p'}
 \]
 for a polynomial $h$ of degree at most $n$, so that
 \[
   \|\bp(\overline{z}f)\|_{H^p} \leq C_p \|S_z(f)\|_p.
 \]
 We will not need this in the future, but we note that its proof
 does not use $f \in L^{2p/(p+1)}$, unlike our proof of the theorem. 
 
We next apply Theorem \ref{thm:bpinhardy} to functions $f$ of the form 
$f = F^{p/2} \overline{F}^{(p/2)-1}$ using Theorem \ref{thm:caldhl}. 
Note that in this case
$\partial_z(f) = \frac{p}{2}|F|^{p-2}F'$.
Now note that $||F|^{p-2}F'| =  |\nabla (|F|^{p-1})|/(p-1)$ and thus 

\[
  S_z(f) \leq 
    \frac{p/2}{p-1} S(|F|^{p/2}).
  \]
\begin{theorem}
  Let $k \in A^q$ be not identically $0$
  and let $F$ be the Bergman space extremal function for
  $k$. 
  If $1 \leq p < \infty$ and $0 < q < \infty$ and $F \in H^q$ then
  $k \in H^{q/(p-1)}$.
\end{theorem}

Combining this with a result in \cite{tjf-ryabpnoteven} yields
the following. 
\begin{corollary}
  If $1 \leq p < \infty$ and $p \leq q < \infty$ then
  $F \in H^q$ if and only if $k \in H^{q/(p-1)}$.
\end{corollary}

\section{Szeg\"{o} Projections and Extremal Problems}

We will state our main result for this section.
  An examination of the proof and the fact that we could compute
  an estimate for the $C_{q,p}$ shows that
  we could compute an estimate for the constants
  $\widetilde{C}_{q,p}$ in the theorem if we wanted.  
  \begin{theorem}\label{thm:hardyextremal}
    Let $1 < p < \infty$ and $1 < q < \infty$.
    Let $\widetilde{C}_{q,p}$ be as defined in equation \eqref{eq:ctilde}.   
    Suppose $\widetilde{C}_{q,p} < 1$.
    Let $k \in H^{p'} \cap H^q$ and let $F$ be the
    $H^p$ extremal function for $k$.  Then $F \in H^{(p-1)q}$ and
    \[
      \|F\|_{(p-1)q}^{p-1} \leq \frac{1}{1-\widetilde{C}_{p,q}}
       \frac{\|k\|_q}{\|k\|_{(H^p)^*}}. 
    \]
  \end{theorem}

  Before proving the theorem, we mention some consequences.
  Let $K$ be the extremal kernel corresponding to $k$ from the
  theory of dual extremal problems.  Then 
  \[
    K = \|k\|_{(H^p)^*} F^{p/2}\overline{F}^{(p/2)-1}.
  \]
  But $K$ is by definition the solution to the problem of finding
  the function of smallest $L^{p'}$ norm of the form
  $\overline{k} - g$, where $g$ is analytic and vanishes at the
  origin. Changing $p'$ to $p$ gives the following result.  
\begin{lemma}
  Let $1 < p < \infty$ and $1 < q < \infty$.  Suppose that
  $\widetilde{C}_{q,p'} < 1$ and that 
  $k \in H^p \cap H^q$ and let $f \in H^p$ satisfy
  $\|f - \overline{k}\|_{H^p} \leq \|g - \overline{k}\|_{H^p}$ for all
  $g \in H^p$ that vanish at the origin.  Then $f \in H^q$ and 
  \[
    \|f-\overline{k}\|_{H^q} \leq \frac{1}{1-\widetilde{C}_{q,p'}} \|k\|_{H^q}
  \]
  and thus
  \[
    \|f\|_{H^q} \leq
    \left(1 + \frac{1}{1-\widetilde{C}_{q,p'}}\right) \|k\|_{H^q}.
   \]
\end{lemma}

Suppose now that $k \in L^p$ and we want to
find the function $f$ that minimizes
$\|f - k\|_p$, where $f$ is in $H^p$.  
This is equivalent to finding the
function $h \in H^p$ which minimizes the norm of
$h - P_S^\perp k$, where $P_S^\perp  = I - P_s$
is the projection of $L^2$ onto $(H^2)^\perp$.
Indeed, we have the relation $f = h + P_S(k)$.
But this is
equivalent to minimizing the norm of 
$g - zP_S^\perp k$, where $g$ is analytic and vanishes at the
origin.  Note that $z P_S^\perp k$ is an anti-analytic function and
that $g = zh$.  
Since then
\[
  \|g-zP_s^\perp k\|_{H^q} \leq
  \left(1 + \frac{1}{1-\widetilde{C}_{q,p'}} \right) \|P_S^\perp k\|_{H^q}
\]
and $\|P_S^{\perp} k\|_q \leq \mathfrak{s}_q \|k\|_q$ and
finally $f = g/z + P_s(k)$ we have 
the following result.  
\begin{theorem}
  Let $1 < p < \infty$ and $1 < q < \infty$.  Suppose that
  $\widetilde{C}_{q,p'} < 1$ and that 
  $k \in L^p \cap L^q$ and let $f \in H^p$ satisfy
  $\|f - k\|_{H^P} \leq \|g - k\|_{H^p}$ for all
  $g \in H^p$.  Then $f \in H^q$ and 
  \[
    \|f\|_{H^q} \leq
    \left(2 + \frac{1}{1-\widetilde{C}_{q,p'}}\right)
      \mathfrak{s}_q \|k\|_{L^q}.
  \]
\end{theorem}

\begin{proof}[Proof of Theorem \ref{thm:hardyextremal}]
Let $F \in H^p$ and let $f = F^{p/2} \overline{F}^{(p/2)-1}$.  
Suppose that 
$h$ is analytic and $0$ at the origin.
For the moment assume that both $F$ and 
$h$ are analytic in the closed unit disc.  Then an application of
Green's formula \cite[VI.3]{Garnett} shows that 
\[
  \frac{1}{2\pi} \int_0^{2\pi} f h \, d \theta
  =
  \frac{1}{2\pi} \int_{\mathbb{D}} \Delta(f h) \log(1/|z|) \, dA.
\]
Let
\[
  A = \frac{1}{2\pi}\int_{\frac{1}{4}\mathbb{D}} \Delta(f h) \log(1/|z|) \, dA
\]
and
\[
  B = \frac{1}{2\pi}\int_{\mathbb{D} \setminus \frac{1}{4} \mathbb{D}}
  \Delta(f h) \log(1/|z|) \, dA.
\]

Now note that $\Delta(f h) = 4(\partial_{\overline{z}} f) h' + \Delta(f) h$.
First note that , 
\[
\partial_{\overline{z}} f = ((p/2)-1)F^{p/2} \overline{F}^{(p/2)-2}
    \overline{F}'.
\]
Now this means that
\[
|\partial_{\overline{z}} f| = |(p/2)-1| |F|^{p-2} |F'|
=
\frac{|p-2|}{2\sqrt{2}(p-1)}| \nabla |F|^{p-1}|.\]
Also, 
\[
|\Delta(f)| = 4\left|\frac{p}{2}-1\right|\left(p/2\right)|F|^{p-3}|F'|^2.
\]
But this equals
\[
  \frac{p|p-2|}{2(p-1)} 
|\nabla |F|^{(p-1)/2}|^2
  \]
  So we have that
\[
  |\Delta(fh)| \leq
  \frac{|p-2|}{{4}(p-1)} |\nabla |F|^{p-1}| \, 
     |\nabla |h|\|  +
     \frac{p|p-2|}{2(p-1)} |\nabla|F|^{(p-1)/2}|^2 |h|.
\]

Because of the above formulas, the integrability of $\log(1/|z|)$ and
the fact that point evaluation (and point evaluation of derivatives)
is a bounded linear functional in all Hardy spaces, 
there is a constant $\mathfrak{k}_q$ depending only on $q$ such that
\[
A \leq \frac{p|p-2|}{p-1} \mathfrak{k}_q \|F\|_{(p-1)q}^{p-1} \|k\|_{q'},
\]
 where the norms
are with respect to Hardy spaces.  
We can also note that
\[
  |B| \leq \frac{2}{2\pi} \int_{\mathbb{D}} |\Delta(fh)| (1-|z|) \, dA
\]
A calculation similar to that in the last section shows that
the right hand side of the above expression equals
\[
  \frac{2}{2\pi} \int_0^{2\pi} \iint_{\Gamma_{\theta}} |\Delta(f h)|
  r \, d\phi \, dr d\theta
\]

Putting this all together gives
\[
  \begin{split}
  \frac{1}{2\pi} \int_0^{2\pi} f h \, d \theta  &\leq 
        \frac{p|p-2|}{p-1}\mathfrak{k}_q \|F\|_{(p-1)q}^{p-1} \|h\|_{q'} +{} \\
      &\qquad \frac{|p-2|}{{4}(p-1)} 
     \|S(|F|^{p-1})\|_q \|S(h)\|_{q'} +{}\\ 
    &\qquad \frac{p|p-2|}{2(p-1)} \|S(|F|^{(p-1)/2})^2\|_{q} \|h^*\|_{q'}
\end{split}
\]
where $h^*$ is the nontangential maximal function. 
Thus
\[
  \begin{split}
    \frac{1}{2\pi} \int_0^{2\pi} f h \, d \theta  \leq 
      \biggl(\frac{p|p-2|}{p-1} \mathfrak{k}_q &+
      \frac{p-2}{{4}(p-1)} C_{q,p-1} C_{q',1} + {} \\
      &\mathfrak{m}_{q'}\frac{p|p-2|}{2(p-1)} C_{2q,(p-1)/2}^2 \biggr)
    \|F\|_{(p-1)q}^{p-1} \|h\|_{q'}
  \end{split}
\]
where $\mathfrak{m}_{q'}$ is defined by
$\|h^*\|_{q'} \leq \mathfrak{m}_{q'} \|h\|$.
Applying the above result to dilations allows us to remove the regularity
assumptions on $f$ and $h$. 

We now show how to apply the above to extremal problems.
It is here we will need the assumption $q > 1$.  
Let $q \geq p$ and suppose that $f \in H^q$ where $q \geq p/(p-1)$.
Note that this is equivalent to $F \in H^{q(p-1)}$.  
Note that if $h$ is analytic and $0$ at the origin that
\[
  \frac{1}{2\pi} \int_0^{2\pi} f h \, d \theta =
  \frac{1}{2\pi} \int_0^{2\pi}(f - P_S(f)) h \, d \theta.
\]
Let      $\mathfrak{s}_p=\csc(\pi/p)$
  be the norm of the Szeg\"{o} projection from $L^p$ to $H^p$. Now, 
\[
\begin{split}
  \|f-P_s(f)\|_q &=
  \sup_{\substack{h \textrm{ harmonic}\\ \|h\|_{q'} \leq 1}}
  \frac{1}{2\pi} \int_0^{2\pi} (f - P_S(f))z h \, d \theta \\
  &=
  \sup_{\substack{h \textrm{ harmonic}\\ \|h\|_{q'} \leq 1}}
  \frac{1}{2\pi} \int_0^{2\pi} (f - P_S(f)) z P_S(h) \, d \theta \\
  &\leq
  \sup_{\substack{h \textrm{ analytic}\\ \|h\|_{q'} \leq \mathfrak{s}_{q'}}}
  \frac{1}{2\pi} \int_0^{2\pi} (f - P_S(f)) zh \, d \theta
\end{split}
\]

Let
\begin{equation}\label{eq:ctilde}
\begin{split}
 \widetilde{C}_{q,p} = 
    \mathfrak{s}_{q'} \biggl(
      \frac{p|p-2|}{p-1} \mathfrak{k}_q &+
      \frac{p-2}{{4}(p-1)} C_{q,p-1} C_{q',1} + {}\\
      &\mathfrak{m}_{q'}\frac{p|p-2|}{2(p-1)} C_{2q,(p-1)/2}^2
      \biggr)
\end{split}
\end{equation}
    where $\mathfrak{m}$ is constant for the nontangential
    maximal operator. 
But our work above now shows that
$\|f - P_S(f)\|_q$ is at most $\widetilde{C}_{q,p} \|f\|_q$.
  If this constant $\widetilde{C}_{q,p}$ is less than $1$, then we have
  $\|P_S f\|_q \geq (1-\widetilde{C}_{q,p}) \|f\|_q$, as long as
  $\|f\|_q$ is finite.

  We can now phrase the above result in terms of extremal problems.  If
  $F$ is the $H^p$ extremal function for $k$ and $\|F\|_{q(p-1)}$ is
  finite then
  \[
    \|F\|_{q(p-1)}^{p-1} \leq \frac{1}{1-\widetilde{C}_{q,p}}
            \frac{\|k\|_q}{\|k\|_{(H^p)^*}}.
  \]

  All that is left to do is show that $F \in H^{(p-1)q}$ if $k \in H^{q}$. 

  To this end, let $k_n$ be a sequence of polynomials approaching
  $k$ in the $H^q$ and $H^{p'}$ norms.
  We may assume that $\|k_n\|_q \leq 2\|k\|_q$ for all $n$
  and $\|k_n\|_{(H^p)^*} \geq \|k\|_{(H^p)^*}/2$ for all $n$.
  It is known that the extremal
  function $F_n$ corresponding to $k_n$ is continuous in the closed
  unit disc.  (In fact, an explicit form is known for such $F_n$
  from the theory of dual extremal problems).
  Thus for any fixed $0 \leq r < 1$ we have 
\[
    M_{q(p-1)}^{p-1}(r,F_n) \leq 
    \|F_n\|_{q(p-1)}^{p-1} \leq \frac{4}{1-\widetilde{C}_{q,p}}
       \frac{\|k\|_q}{\|k\|_{(H^p)^*}}.
\]
It is known that $F_n \rightarrow F$ in $H^p$ \cite{tjf1}, and thus
$F_n \rightarrow F$ uniformly on compact subsets of the unit disc.
This means that $M_{q(p-1)}(r,F_n) \rightarrow M_{q(p-1)}(r,F)$, and so
\[
    M_{q(p-1)}^{p-1}(r,F) \leq 
    \frac{4}{1-\widetilde{C}_{q,p}}
      \frac{\|k\|_q}{\|k\|_{(H^p)^*}}. 
\]  
Since $r$ is arbitrary, this proves that $F \in H^{q(p-1)}$. 
\end{proof}

\providecommand{\bysame}{\leavevmode\hbox to3em{\hrulefill}\thinspace}
\providecommand{\MR}{\relax\ifhmode\unskip\space\fi MR }
\providecommand{\MRhref}[2]{%
  \href{http://www.ams.org/mathscinet-getitem?mr=#1}{#2}
}
\providecommand{\href}[2]{#2}


\begin{thebibliography}{10}

\bibitem{D_Hp}
Peter Duren, \emph{Theory of {$H\sp{p}$} spaces}, Pure and Applied Mathematics,
  Vol. 38, Academic Press, New York, 1970. \MR{0268655 (42 \#3552)}

\bibitem{tjf1}
Timothy Ferguson, \emph{Continuity of extremal elements in uniformly convex
  spaces}, Proc. Amer. Math. Soc. \textbf{137} (2009), no.~8, 2645--2653.
  \MR{2497477}

\bibitem{tjf2}
Timothy Ferguson, \emph{Extremal problems in {B}ergman spaces and an extension
  of {R}yabykh's theorem}, Illinois J. Math. \textbf{55} (2011), no.~2,
  555--573 (2012). \MR{3020696}

\bibitem{tjf-ryabpnoteven}
\bysame, \emph{Extremal problems in {B}ergman spaces and an extension of
  {R}yabykh's {H}ardy space regularity theorem for {$1<p<\infty$}}, Indiana
  Univ. Math. J. \textbf{66} (2017), no.~1, 259--274. \MR{3623410}

\bibitem{Garnett}
John~B. Garnett, \emph{Bounded analytic functions}, Pure and Applied
  Mathematics, vol.~96, Academic Press Inc. [Harcourt Brace Jovanovich
  Publishers], New York, 1981. \MR{628971 (83g:30037)}

\bibitem{KhavinsonSYa1951Trans}
S.~Ya. Havinson, \emph{On some extremal problems of the theory of analytic
  functions}, Amer. Math. Soc. Transl. (2) \textbf{32} (1963), 139--154.

\bibitem{KhavinsonSYa1949}
S.~Ya. Khavinson, \emph{On an extremal problem of the theory of analytic
  functions}, Uspehi Matem. Nauk (N.S.) \textbf{4} (1949), no.~4(32), 158--159.
  \MR{0033887 (11,508e)}

\bibitem{KhavinsonSYa1951}
\bysame, \emph{On some extremal problems of the theory of analytic functions},
  Moskov. Gos. Univ. U\v cenye Zapiski Matematika \textbf{148(4)} (1951),
  133--143. \MR{0049322 (14,155f)}

\bibitem{RogosinskiShapiro}
W.~W. Rogosinski and H.~S. Shapiro, \emph{On certain extremum problems for
  analytic functions}, Acta Math. \textbf{90} (1953), 287--318. \MR{0059354}

\bibitem{Ryabykh}
V.~G. Ryabykh, \emph{Extremal problems for summable analytic functions},
  Sibirsk. Mat. Zh. \textbf{27} (1986), no.~3, 212--217, 226 ((in Russian)).
  \MR{853902 (87j:30058)}

\bibitem{Dragan_isoperimetric}
Dragan Vukoti\'c, \emph{The isoperimetric inequality and a theorem of {H}ardy
  and {L}ittlewood}, Amer. Math. Monthly \textbf{110} (2003), no.~6, 532--536.
  \MR{1984405}

\end{thebibliography}
\end{document}